\documentclass[11pt,twoside]{amsart}
\usepackage{amssymb,epsfig,amsmath,latexsym,amscd,amsthm,hyperref}
\usepackage{tikz}
\usepackage{enumitem}
\usepackage{comment}
\usepackage{subfig}
\usepackage{xcolor}
\usepackage{float}
\usepackage{tcolorbox}

\usepackage{graphicx}

\nonstopmode

\numberwithin{equation}{section}
\font\tengothic=eufm10 scaled\magstep 1
\font\sevengothic=eufm7 scaled\magstep 1
\newfam\gothicfam
\textfont\gothicfam=\tengothic
\scriptfont\gothicfam=\sevengothic

\DeclareMathOperator{\pnt}{\raise 0.5mm \hbox{\large\bf.}}

\newtheorem{theorem}{Theorem}[section]
\newtheorem{lemma}[theorem]{Lemma}
\newtheorem{proposition}[theorem]{Proposition}
\newtheorem{corollary}[theorem]{Corollary}

\theoremstyle{definition}
\newtheorem{definition}[theorem]{Definition} 
\newtheorem{remark}[theorem]{Remark}
\newtheorem{example}[theorem]{Example}

\newtheorem{notation}[theorem]{Notation}

\definecolor{blue-violet}{rgb}{0.54, 0.17, 0.89}
\definecolor{darkgreen}{rgb}{0.01, 0.75, 0.24}

\author[G.~Favacchio]{Giuseppe Favacchio}
\address[G.~Favacchio]{DISMA-Department of Mathematical Sciences \\
	Politecnico di Torino, Turin, Italy}
\email{giuseppe.favacchio@polito.it}

\author[J.~Migliore]{Juan Migliore}
\address[J.~Migliore]{Department of Mathematics, University of Notre Dame, Notre Dame, IN 46556}
\email{migliore.1@nd.edu}

\title[ACM property for unions of lines]{The ACM property for unions of lines  in~$\mathbb P^1 \times \mathbb P^2$}

\begin{document}
\thanks{Version: \today }	
	
\begin{abstract}
This paper examines the Arithmetically Cohen-Macaulay (ACM) property for certain codimension~2 varieties in $\mathbb P^1\times \mathbb P^2$ called \textit{sets of lines} in $\mathbb P^1\times \mathbb P^2$ (not necessarily reduced). We discuss some obstacles to finding a general characterization. We then consider certain classes of such curves, and we address two questions. First, when are they themselves ACM? Second, in a non-ACM reduced configuration, is it possible to replace one component of a primary (prime) decomposition by a suitable power (i.e. to ``fatten" one line) to make the resulting scheme~ACM? Finally, for our classes of such curves, we characterize the locally Cohen-Macaulay property in combinatorial terms by introducing the definition of a {\it fully v-connected} configuration. We apply some of our results to give analogous ACM results for sets of lines in $\mathbb P^3$.
\end{abstract}

\maketitle


\section{Introduction}

	It is still an open problem to determine a geometric characterization of the arithmetically Cohen-Macaulay (ACM) property for varieties in a multiprojective space.  While this problem has strong connections to the analogous problem for varieties in a projective space, there are also striking differences. To give just two simple illustrations, any finite set of points in a projective space is automatically ACM, while this is not true in any multiprojective space. Indeed, a whole book has been devoted to this topic \cite{GV-book} just for the case of $\mathbb P^1 \times \mathbb P^1$. Secondly, the structure of a multihomogeneous ideal differs from that of a homogeneous ideal in important ways, and this contributes to the difficulty. 
	
	In both the projective and the multiprojective settings, the problem can have a strong combinatorial  component in addition to geometric and algebraic ones. One specific problem that has been studied in the projective setting in many different ways is the question of when a union of lines is ACM. 
	This is still wide open in projective spaces, in general. In this paper we begin the study of this problem in a specific multiprojective space, namely $\mathbb P^1 \times \mathbb P^2$. 
	
	Included in this problem (in either setting) is the question of what can be added to a non-ACM variety to make it become ACM, hopefully adding as little as possible. We  propose a new variation of this question (and solve it in our special setting): given a reduced configuration of lines that is not ACM, can one ``fatten" one of the components and make the new scheme ACM? Here, by ``fatten" we mean that in a primary decomposition we replace a prime ideal $\mathfrak p$, corresponding to one of the lines, by a power $\mathfrak p^k$ for suitable $k$. (Notice that $\mathfrak p$ is a complete intersection, so $\mathfrak p^k$ is unmixed.) After solving this problem in the situation of this paper, we give as a corollary the analogous result for unions of lines in $\mathbb P^3$. (See Theorem \ref{p. fully v-connected 1} for $\mathbb P^1 \times \mathbb P^2$ and Corollary \ref{cor:curve P3} for $\mathbb P^3$.) We also explore the property of being locally Cohen-Macaulay, giving a characterization for our curves in $\mathbb P^1 \times \mathbb P^2$.

	Many papers in the literature have studied the ACM property for different kinds of subvarieties of multiprojective spaces, especially for sets of points. 
	Despite the fact that in $\mathbb P^1\times \mathbb P^1$ there are several characterizations of the ACM property (again, see \cite{GV-book} for a detailed discussion of the topic) only a few other results are known in general. In particular, a characterization was given in  \cite{FGM2018} of the ACM property for sets of points in $(\mathbb P^1)^r=\mathbb P^1\times\cdots\times\mathbb P^1,$ and in \cite{FM2019} the authors described, under certain conditions, the ACM property for sets of points in $\mathbb P^m\times \mathbb P^n$ and, with more details, in $\mathbb P^1\times \mathbb P^n.$ See also \cite{GV08} for some other results in this direction.


A crucial difference in the study of the ACM property of a set of points in $\mathbb P^1\times \mathbb P^1$  and a set of points in any other multiprojective space is given by the codimension. A set of points has codimension two  in $\mathbb P^1\times \mathbb P^1$, and  strictly larger than two in a different multiprojective space.
So, it is interesting to approach the study of the ACM property by looking at the codimension two varieties more generally. The paper \cite{FGP}   investigates the ACM property for  2-codimensional varieties in $\mathbb P^1\times \mathbb P^1\times \mathbb P^1$. Such varieties have a different nature and meaning from sets of points in $(\mathbb P^1)^r$, $r\ge2$, but the characterization of the ACM property deeply uses a common fact.  All of them are defined by ideals generated by particular products of linear forms.  In this case, this peculiarity makes the description of the ACM property merely combinatorial.

The  vast nature of the problem  draws in many standard tools and techniques  from the homogeneous setting. These include, just to cite some of them, hyperplane sections, basic double G-linkage, liaison addition, and liaison. On the other hand,
the study of varieties in multiprojective spaces plays an important role in several branches of mathematics, and it finds an application in different contexts;  these include the study of monomial ideals (see for instance \cite{AFH, Nem}), scrolls~(\cite{ES}), symbolic powers (\cite{CFGetc, GHVT1,GHVT2}), tensor analysis (\cite{CS}) and virtual resolutions~(\cite{BES,GLLM}), just to give a partial list.


In this paper we call \textit{sets of lines} of $\mathbb P^1\times \mathbb P^2$ certain codimension two varieties in $\mathbb P^1\times \mathbb P^2$.   We will also look at these varieties as unions of planes in $\mathbb P^4$. This is possible since a bihomogeneous linear form of $k[\mathbb P^1\times \mathbb P^2]$ also defines a hyperplane in $\mathbb P^4$. Thus, a line in $\mathbb P^1\times \mathbb P^2$ can be viewed as a plane (the intersection of two hyperplanes) of $\mathbb P^4.$
This correspondence allows us to move the study of several properties, including the  Cohen-Macaulay property, from $\mathbb P^1\times \mathbb P^2$ to $\mathbb P^4.$ 
The converse is not always possible for two reasons due to the nature of the problem. 
First, a linear form in $k[\mathbb P^4]$ is not necessarily an element of  $k[\mathbb P^1\times \mathbb P^2]_{(1,0)}$ or $k[\mathbb P^1\times \mathbb P^2]_{(0,1)}$. Moreover, a plane in $\mathbb P^4$ could be defined by an ideal that is, even if bihomogeneous, not saturated in  $\mathbb P^1\times \mathbb P^2.$ See Remark \ref{P4 connection} for more detail about this connection.

From what has been  observed above, we only have lines of two different types in  $\mathbb P^1\times \mathbb P^2$.
The \textit{horizontal lines}, given by the intersection of a hyperplane of degree~$(1,0)$ with one of degree~$(0,1)$,  and the \textit{vertical lines}, that are intersections of two hyperplanes of degree $(0,1)$.
The latter gives an immediate relation with $\mathbb P^2$ since a set of vertical lines  in $\mathbb P^1 \times \mathbb P^2$ corresponds, in a way made precise in Remark \ref{cone remark}, to a cone over a set of points in $\mathbb P^2$.

 The main result in Section \ref{s.P2}, Proposition \ref{p.points in P2 and h-vector}, concerns sets of points in $\mathbb P^2$, which will be used later in the paper. Given a set of (fat) points $Y $and a curve $\mathbb V(F)$ passing through some of them, we give a characterization, in terms of the $h$-vector of a certain subscheme of $Y$, to  guarantee that $I_Y+(F)$ is a saturated ideal. In   Section \ref{sec:preliminaries} we introduce the notation and examine some obstacles to the characterization of the ACM property; see Remark \ref{r. liaison} and Example \ref{e. liaison}.
In  Section \ref{s. starting case} we apply Proposition \ref{p.points in P2 and h-vector} to study the ACM property  for sets of lines in $\mathbb P^1 \times \mathbb P^2$ satisfying a particular condition introduced in Notation \ref{notation}. Precisely, we assume that in the sets under consideration the horizontal lines are reduced and  two different horizontal lines of $Z$ are not contained in a hyperplane defined by a form of degree $(0,1)$. 

  We begin with an additional assumption (see Theorem~\ref{t. char ACM 1 plane} and Remark \ref{using result insection 2}) and then show that this can be extended (Proposition \ref{d. hat Z},  Proposition \ref{p. Z ACM -> hat Z ACM} and Corollary \ref{c.Betti ACM}). We then apply these results to a class of configurations of lines in $\mathbb P^1 \times \mathbb P^2$ where only one is non-reduced. In this situation we give a necessary and sufficient condition for ACMness in terms of the multiplicity of the non-reduced component.  As a corollary we give a result for $\mathbb P^3$, as mentioned above. Our most general result (still with some assumptions) gives a characterization for when the configuration is locally Cohen-Macaulay (see Theorem \ref{t. local CM <-> fully v connected}).


\vspace{.1in}

\noindent{\bf Acknowledgments.} Many of the results in this paper were inspired through computer experiments using  CoCoA \cite{cocoa} and Macaulay2 \cite{macaulay2}. This work was done while the first author was partially supported by Università degli studi di Catania, piano della ricerca PIACERI 2020/22 linea intervento 2 and by the National Group for Algebraic and Geometrical Structures and their Applications (GNSAGA-INdAM)  and the second author was partially supported by a Simons Foundation grant (\#309556). We thank the referees for the useful comments and suggestions. 

 
\section{A preliminary result about saturation in $\mathbb P^2$}\label{s.P2} 

In this section we give a criterion, in terms of the $h$-vector, to establish if an ideal of fat points plus a form is saturated or not.
In the next sections, we will relate it to the ACM property for some sets of lines in $\mathbb P^1 \times \mathbb P^2$; see Remark \ref{using result insection 2}. There is an immediate connection with points in $\mathbb P^2$ since a set of vertical lines  in $\mathbb P^1 \times \mathbb P^2$ is the cone over a set of points in $\mathbb P^2$.

We start the section by recalling some standard notation. 
Given a finite set of $n$ distinct points $W=\{P_1,\ldots, P_n\}$ in $\mathbb P^N$,   and $m_1,\ldots, m_n$ positive integers, we write $Y=m_1P_1+\cdots+m_nP_n$ for the set of fat points defined by the saturated homogeneous ideal
$$I_Y=\bigcap_{P_i\in W} (I_{P_i})^{m_i}\subseteq S=k[\mathbb P^N].$$ 
The degree of $Y$ is $\deg(Y)=\sum \binom{m_i +N-1}{ N-1}.$
Recall that for a zero-dimensional scheme  $Y\subset \mathbb P^N$ the Hilbert function of $Y$ is defined as the numerical function $H_Y: \mathbb N \to \mathbb N$ such that
\[
H_Y(i)=\dim_k(S/I_Y)_i= \dim_k S_i-\dim_k(I_Y)_i.
\]
Since $H_Y(\tau)=\deg(Y)$ for $\tau$ large enough, the first difference of the Hilbert function $\Delta H_Y(i)=H_Y(i)-H_Y(i-1)$ is eventually  zero. The {\em $h$-vector} of $Y$ is 
\[
h_Y= h = (1,h_1,\ldots, h_d )
\]
where $h_i  = \Delta H_Y(i)$ and $d$ is the last index such that $\Delta H_Y(i)  > 0$.

\begin{proposition}\label{p.points in P2 and h-vector}
	Let $Y=m_1P_1+m_2P_2+ \cdots m_nP_n$ be a set of (fat) points in $\mathbb P^2.$ Let $\mathcal C=\mathbb{V}(F)\subseteq \mathbb P^2$ be a curve defined by a form $F$ of degree $\deg(F)=d.$  For any $i$, assume that if $F$ vanishes at $P_i \in Y$, then the vanishing order is at least $m_i$,  so $F \in [I_{m_iP_i}]_d$.
	Let $Y_1$ and $Y_2$ be the two  zero-dimensional schemes  such that $I_{Y_1}=(I_Y+F)^{sat}$ and $I_{Y_2}=I_Y:F.$ Then $Y=Y_1\cup Y_2$, where $Y_1 \cap Y_2 = \emptyset$, and  $I_Y+(F)\subseteq k[\mathbb P^2]$ is saturated  if and only if, for each $\tau\ge 0,$ we have $$h_{Y}(\tau)=h_{Y_1}(\tau)+h_{Y_2}(\tau-d).$$
 \end{proposition}
\begin{proof}With these hypotheses we have $Y_1=Y\cap \mathcal C$ and $Y_2=Y\setminus \mathcal C$, and $Y_1 \cap Y_2 = \emptyset$ (the latter is because the multiplicity of $F$ at each point is big enough). 
	Note that $I_{Y_2}=I_{Y}:F$ is automatically saturated since $I_Y$ is saturated. On the other hand,  $I_{Y_1}=(I_Y+(F))^{sat}$ but $I_Y + (F)$ is not necessarily already saturated. Consider the short exact sequence 
	\[
	0 \to \dfrac{R}{I_Y:F}(-d)\stackrel{\cdot F}{\longrightarrow} \dfrac{R}{I_{Y}}\to \dfrac{R}{I_Y+(F)}\to 0. 
	\]
In any degree $\tau$, the first vector space in the sequence has dimension $h_{Y_2}(\tau - d)$ and the second has dimension $h_Y(\tau)$ since the corresponding ideals are saturated. Then the third ideal, $I_Y + (F)$ is saturated if and only if the third vector space has dimension $h_{Y_1}(\tau)$, since the latter Hilbert function uses the saturated ideal of $Y_1$, which is $(I_Y + (F))^{sat}$.
\end{proof}

\begin{remark}In Proposition \ref{p.points in P2 and h-vector}, if $Y$ is a reduced set of points then the condition about the vanishing order of $F$ is automatically satisfied.  
\end{remark} 

In the next example we show different ways to add points, in a ``non-saturated ideal case", in order to construct a saturated ideal.

\begin{example}
	Denote by $\ell_i$ the line of $\mathbb P^2$ defined by the linear form $x+iy,\ 1 \leq i \leq 4$. Let $\mathcal C$ be the curve given by  the union of the lines $\ell_1,\ell_2,\ell_3, \ell_4$, and let $F$ be a form defining $\mathcal C$. Let $Y$ be a set of five distinct points (no three on a line) such that only four of them belong to $C$.
	Say for instance $P_1=(1,-1,1)$, $P_2=(2,-1,1)$, $P_3=(6,-2,1)$, $P_4=(8,-2,1)$ and $P_5=(1,1,1).$

	\begin{center}
		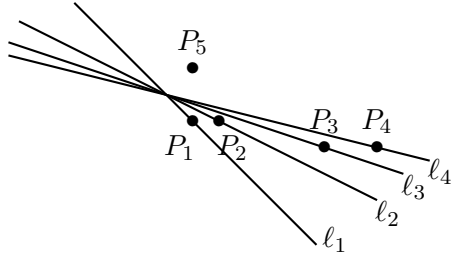
\begin{figure}[H]
			\centering
			\begin{tikzpicture}[scale=0.35]

			\draw [thick] (-3.5,3.5) - - (5.7,-5.7);
			\node  at (6.4,-5.5) { $\ell_1$};
			\draw [thick] (8,-4) - - (-5.6,2.8);
			\node  at (8.4,-4.5) { $\ell_2$};
			\draw [thick] (9,-3) - - (-6,2);
			\node  at (9.4,-3.5) { $\ell_3$};
			\draw [thick] (10,-2.5) - - (-6,1.5);
			\node  at (10.4,-2.9) { $\ell_4$};
			\node  at (1,-1) { $\bullet$};
			\node  at (0.5,-2) { $P_1$};
			\node  at (2,-1) { $\bullet$};
			\node  at (2.5,-2) { $P_2$};
			\node  at (6,-2) { $\bullet$};
			\node  at (6,-1) { $P_3$};
			\node  at (8,-2) { $\bullet$};
			\node  at (8,-1) { $P_4$};
			\node  at (1,1) { $\bullet$};
			\node at (1,2){$P_5$};
			\end{tikzpicture} 
			\caption{Five points, four of them on a quartic.} 
		\end{figure}
	\end{center}
Using the notation in Proposition \ref{p.points in P2 and h-vector}, we set  $Y_1=\{P_1, P_2,P_3,P_4\}$ and $Y_2=\{P_5\}.$ So we~have
	\[
	\begin{array}{c|c|c|c|c|ccc}
	h_Y& 1& 2& 2& 0& 0 \\
	\hline
	h_{Y_1}& 1& 2& 1& 0& 0 \\
	\hline
	h_{Y_2}& 0& 0& 0& 0& 1 \\	
	\end{array}
	\] 
	where  the third line in the table is the $h$-vector of $Y_2$ shifted by 4 (the degree of $F$). 
	By Proposition \ref{p.points in P2 and h-vector}, we note that $(F)+I_Y$ is not a saturated ideal. Now, by adding points to~$Y_1$, we show three different ways to make the resulting ideal  $(F)+I_Y$ saturated.
	
	\begin{itemize}
		\item[$i)$] 
		We add to $Y$ six points lying on $\mathcal C,$ consisting of two general points in  $\ell_1$ and $\ell_2$ and one general point in $\ell_3$ and $\ell_4$. (Note that we have a similar situation occurs by adding a total of 7, 8, 9 or 10 general points on $\ell_1,\ldots, \ell_4.$ )

\noindent\begin{minipage}{0.3\textwidth}
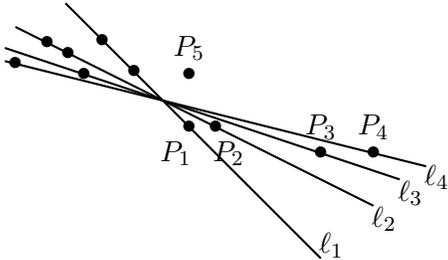
		\begin{figure}[H]
		\centering
		\begin{tikzpicture}[scale=0.35]
			\draw [thick] (-3.7,3.7) - - (6,-6);
\node  at (6.4,-5.5) { $\ell_1$};
\draw [thick] (8,-4) - - (-5.6,2.8);
\node  at (8.4,-4.5) { $\ell_2$};
		\draw [thick] (9,-3) - - (-6,2);
		\node  at (9.4,-3.5) { $\ell_3$};
		\draw [thick] (10,-2.5) - - (-6,1.5);
		\node  at (10.4,-2.9) { $\ell_4$};
		\node  at (1,-1) { $\bullet$};
		\node  at (0.5,-2) { $P_1$};
		\node  at (2,-1) { $\bullet$};
		\node  at (2.5,-2) { $P_2$};
		\node  at (6,-2) { $\bullet$};
		\node  at (6,-1) { $P_3$};
		\node  at (8,-2) { $\bullet$};
		\node  at (8,-1) { $P_4$};
		\node  at (1,1) { $\bullet$};
		\node at (1,2){$P_5$};
		
		\node  at (-1.1,1.1) { $\bullet$};
		\node  at (-2.3,2.3) { $\bullet$};
		\node  at (-4.4,2.2) { $\bullet$};
		\node  at (-3.6,1.8) { $\bullet$};
		\node  at (-3,1) { $\bullet$};
		\node  at (-5.6,1.4) { $\bullet$};
		\end{tikzpicture} 
	\end{figure}
\end{minipage}%
\hfill%
\begin{minipage}{0.6\textwidth}\centering	\ \ \ So, we have \ \ \
	$
	\begin{array}{c|c|c|c|c|ccc}
	h_Y& 1& 2& 3& 4& 1 \\
	\hline
	h_{Y_1}& 1& 2& 3& 4& 0 \\
	\hline
	h_{Y_2}& 0& 0& 0& 0& 1 \\	
	\end{array}
	$
\end{minipage}

		\item[$ii)$] We add to $Y$ the triple point defined by the ideal $(x,y)^3$. The vanishing order of $F$ at this point is $4.$

\noindent\begin{minipage}{0.3\textwidth}
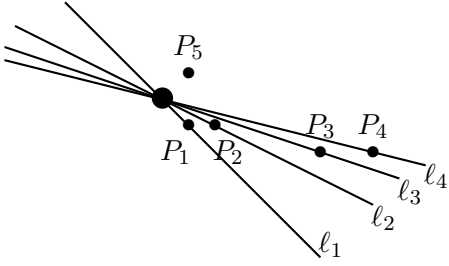
			\begin{figure}[H]
	\centering
	\begin{tikzpicture}[scale=0.35]
			\draw [thick] (-3.7,3.7) - - (6,-6);
\node  at (6.4,-5.5) { $\ell_1$};
\draw [thick] (8,-4) - - (-5.6,2.8);
\node  at (8.4,-4.5) { $\ell_2$};
	\draw [thick] (9,-3) - - (-6,2);
	\node  at (9.4,-3.5) { $\ell_3$};
	\draw [thick] (10,-2.5) - - (-6,1.5);
	\node  at (10.4,-2.9) { $\ell_4$};
	\node  at (1,-1) { $\bullet$};
	\node  at (0.5,-2) { $P_1$};
	\node  at (2,-1) { $\bullet$};
	\node  at (2.5,-2) { $P_2$};
	\node  at (6,-2) { $\bullet$};
	\node  at (6,-1) { $P_3$};
	\node  at (8,-2) { $\bullet$};
	\node  at (8,-1) { $P_4$};
	\node  at (1,1) { $\bullet$};
	\node at (1,2){$P_5$};
	
	\node  at (0,0) { \huge $\bullet$};
	\end{tikzpicture} 
\end{figure}
\end{minipage}%
\hfill%
\begin{minipage}{0.6\textwidth}\centering	\ \ \ So, we get \ \ \
$
\begin{array}{c|c|c|c|c|ccc}
h_Y& 1& 2& 3& 4& 1 \\
\hline
h_{Y_1}& 1& 2& 3& 4& 0 \\
\hline
h_{Y_2}& 0& 0& 0& 0& 1 \\	
\end{array}
$
\end{minipage}		
		
		\item[$iii)$] Let $\ell$ be a line containing $P_5,$ say for instance $\ell$ defined by the linear form $y-z$, and add to $Y$ the four intersection points of $\ell$ and $\mathcal C$ i.e.  $P_6=(-1,1,1),$ $P_7=(-2,1,1),$  $P_8=(-3,1,1),$ $P_9=(-4,1,1).$ 

\noindent\begin{minipage}{0.3\textwidth}
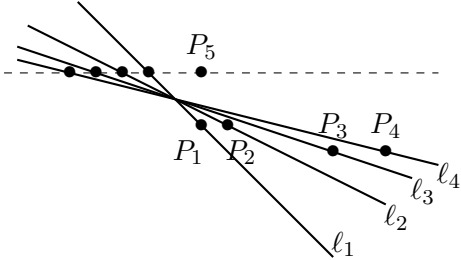
				\begin{figure}[H]
	\centering
	\begin{tikzpicture}[scale=0.35]
			\draw [thick] (-3.7,3.7) - - (6,-6);
\node  at (6.4,-5.5) { $\ell_1$};
\draw [thick] (8,-4) - - (-5.6,2.8);
\node  at (8.4,-4.5) { $\ell_2$};
	\draw [thick] (9,-3) - - (-6,2);
	\node  at (9.4,-3.5) { $\ell_3$};
	\draw [thick] (10,-2.5) - - (-6,1.5);
	\node  at (10.4,-2.9) { $\ell_4$};
	\node  at (1,-1) { $\bullet$};
	\node  at (0.5,-2) { $P_1$};
	\node  at (2,-1) { $\bullet$};
	\node  at (2.5,-2) { $P_2$};
	\node  at (6,-2) { $\bullet$};
	\node  at (6,-1) { $P_3$};
	\node  at (8,-2) { $\bullet$};
	\node  at (8,-1) { $P_4$};
	\node  at (1,1) { $\bullet$};
	\node at (1,2){$P_5$};
	
	\node  at (-1,1) { $\bullet$};
	\node  at (-2,1) { $\bullet$};
	\node  at (-3,1) { $\bullet$};
	\node  at (-4,1) { $\bullet$};
	\draw [dashed] (-6.5,1) - - (10,1);			
	\end{tikzpicture} 
\end{figure}
\end{minipage}%
\hfill%
\begin{minipage}{0.6\textwidth}\centering	\ \ \ So, we have \ \ \
	$
		\begin{array}{c|c|c|c|c|ccc}
h_Y& 1& 2& 3& 2& 1\\
\hline
h_{Y_1}& 1& 2& 3& 2& 0 \\
\hline
h_{Y_2}& 0& 0& 0& 0& 1 \\	
\end{array}
	$
\end{minipage}	
\end{itemize}
\end{example}

\begin{remark}
		\begin{itemize}
			\item If $I_{Y}+F$ is saturated then $F$ is a ``separator" for $Y_2$, i.e. the form $F$  vanishes at $Y\setminus Y_2$ and it doesn't vanish at any point of $Y_2$; see for instance the paper \cite{O} where it first appears.  So, the ideal defining $Y_1$  is obtained  just adding $F$ to  $I_{Y}$.

			\item Let $Y$ be a reduced complete intersection of two planar curves of degrees $d_1 \geq 1$ and $d_2 \geq 1$. Assume that $Y_2$ is a single point, and that $F$ vanishes on $Y_1$ but not $Y_2$. If $I_{Y}+(F)$ is saturated then $F$ cannot have degree $d_1 + d_2-3$, because by the Cayley-Bacharach theorem, any  $F$ vanishing on $Y_1$ must also vanish on $Y_2$.
			
			\item  Using liaison, this kind of remark can be strengthened. For example, let $Y$ be a reduced complete intersection in linear general position and $Y_2$ consists of three points, if $I_{Y}+(F)$ is saturated  then $F$ cannot have degree $d_1 + d_2 - 4$, by a similar analysis. 
		\end{itemize}  
\end{remark}


\section{Notation, terminology and examples}\label{sec:preliminaries}

We work over a field of characteristic zero.
For a product of two projective spaces we define $V = \mathbb P^{a_1} \times  \mathbb P^{a_2}$ and 
\[
\pi_i : V  \rightarrow \mathbb P^{a_i}
\] 
to be the projection to the $i$-th component ($i = 1,2$).

Let $\{ \underline{e}_1, \underline{e}_2 \}$ be the standard basis of $\mathbb N^2$.
Let $x_{i,j}$, with $1 \leq i \leq 2$ and $0 \leq j \leq a_i$ for all $i,j$, be the variables for  $\mathbb P^{a_1}$ and $\mathbb P^{a_2}$. Let 
\[
R = k[x_{1,0}, \dots,x_{1,a_1}, x_{2,0}, \dots, x_{2,a_n}]=k[V],
\]
where the degree of $x_{i,j}$ is $\underline{e}_i$.

\begin{remark} \label{P4 connection}
A subscheme, $X$, of $V$ is defined by a bihomogeneous ideal $I_X$ that is saturated  with respect to $(x_{1,0}, \dots,x_{1,a_1})$ and to $(x_{2,0}, \dots, x_{2,a_n})$. The ideal $I_X$ is generated by a system of multihomogeneous polynomials in $R$.

Given a scheme $X\subseteq V$, its defining ideal, $I_X$, is also saturated with respect to the ideal $(x_{1,0}, \dots,x_{1,a_1}, x_{2,0}, \dots, x_{2,a_n})$. Thus, the ideal $I_X$  defines a different scheme $X'$ in $\mathbb P^{a_1+a_2+1}$. The converse is, of course, not true:  consider for instance the ideal $(x_{1,0}, \dots,x_{1,a_1})$, which defines a subscheme of $\mathbb P^{a_1 + a_2-1}$ but not of $\mathbb P^{a_1} \times \mathbb P^{a_2}$.  

Moreover, the maps in a minimal free resolution of $R/I_X$ can be seen as those of a minimal free resolution of $k[x_{1,0}, \dots,x_{1,a_1}, x_{2,0}, \dots, x_{2,a_n}]/I_{X'}$. In fact, the only difference consists in the degree of the forms appearing as entries of the matrices and not in the forms themselves -- they have  bidegree $(a,b)$ in the first case and degree $a+b$ in the second.   
So, many homological invariants of $X$ and $X'$ are strictly connected and some of them, such as the projective dimension, agree. 
This idea has often been exploited in the literature, for example in   \cite{CFGetc, FGM2018, FM2019, GHVT1}. See also \cite{GV-book}, Remark 3.3, where they say (in the context of $\mathbb P^1 \times \mathbb P^1$) that ``our study of points in $\mathbb P^1 \times \mathbb P^1$ can be seen as an investigation of these special unions of lines in $\mathbb P^3$."
\end{remark}

We say that $X$ is {\em arithmetically Cohen-Macaulay (ACM)} if $R/I_X$ is a Cohen-Macaulay ring.

Let $N = a_1 + a_2 + 2$.
Given a subscheme $X$ of $V$ together with its homogeneous ideal $I_X$, we can also consider the subscheme $\bar X$ of $\mathbb P^{N-1}$ defined by $I_X$. Notice that if $X$ is a zero-dimensional subscheme of $V$, then $I_X$  almost never defines a zero-dimensional subscheme of $\mathbb P^{N-1}$.


In the following, to shorten the notation, we write $R=k[s,t, x,y,z]=k[\mathbb P^1\times \mathbb P^2]$.
We will use the letters ``$A, A', A_i, \ldots $" to denote the elements in $R_{(1,0)}$, and the letters ``$B, B', B_j, \ldots$" for  elements in $R_{(0,1)}$. These bihomogeneous linear forms define hyperplanes in $\mathbb P^1\times\mathbb P^2$.
\begin{definition}
	We call a \textit{line} in $\mathbb P^1\times \mathbb P^2$ the intersection of two hyperplanes defined by a saturated ideal. 
	More precisely, we only have two different kind of lines, depending on the bidegree of the hyperplanes defining it: we call a \textit{horizontal line} a line of type $\mathbb V(A,B)$ and a  \textit{vertical line} a line of type $\mathbb V(B,B')$.
\end{definition}

\begin{remark} \label{cone remark}
Note that two different vertical lines are disjoint in $\mathbb P^1\times \mathbb P^2$. However, there is a useful geometric connection to $\mathbb P^4$. A vertical line is defined by two linear forms in $x,y,z$. Thus the ideal defines, in $\mathbb P^4$, a 2-dimensional linear space that contains the line $\lambda$ defined by $x=y=z=0$. Let $X$ be a union of vertical lines in $\mathbb P^1 \times \mathbb P^2$. The projection of $X$ to $\mathbb P^2$ (i.e. to the second component) is a finite set of points in $\mathbb P^2$. In $\mathbb P^4$ the same ideal $I_X$ (viewed only with the standard grading) defines a finite union of 2-planes whose intersection with the 2-plane defined by $s=t=0$ is also a finite set of points,  a copy of the projection of $X$. In this sense we can view $I_X$ as defining a cone in $\mathbb P^4$ over a finite set of points in $\mathbb P^2$, whose vertex is the line $\lambda$ (as opposed to being a point) and whose components are 2-planes (as opposed to lines). So the cone notion is not useful in $\mathbb P^1 \times \mathbb P^2$, but is useful if we are viewing the corresponding subschemes of $\mathbb P^4$. As a result we sometimes refer to an ideal like $I_X$ as a ``cone ideal."
\end{remark}

Note that the ideal generated by two different forms $A,A'$ is not saturated in $\mathbb P^1\times\mathbb P^2$.  (For instance, $\mathbb V(s,t) = \emptyset$.)
Whenever it is not necessary to make the type of line explicit, we just use the letter $L.$ 
 A \textit{set of lines}~$X$ is a finite collection of lines. We will write either
$X=L_1+L_2+\cdots+L_n\subseteq \mathbb P^1\times\mathbb P^2$ 
or
$X=\{L_1,L_2,\ldots,L_n\}\subseteq \mathbb P^1\times\mathbb P^2.$

There is a natural partition on a set of lines $X$: we will write $X=X_1\cup X_2$ to denote the partition of $X$ into horizontal and vertical lines respectively.

\begin{example} \label{three lines}
	Let $X$ be the set of lines in $\mathbb P^1 \times \mathbb P^2$
	$$X= \mathbb V(s,x)\cup \mathbb V (t,y) \cup \mathbb V(x,y).$$
Then $X$ consists of two horizontal lines, that are $\mathbb V(s,x)$ and $\mathbb V (t,y)$, and one vertical line $\mathbb V(x,y)$. In Figure \ref{fig0} a representation of such set.	
	\begin{center}
		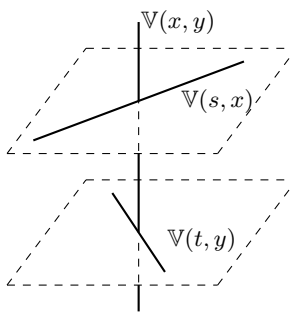
\begin{figure}[H]
			\centering
			\begin{tikzpicture}[scale=0.35]
			\draw[dashed] (0,5) - -  (8,5);  
			\draw[dashed] (3,9) - - (11,9);   
			\draw[dashed] (0,5) - -  (3,9);  
			\draw[dashed] (8,5) - - (11,9); 
			\node [font=\footnotesize] at (7.4,6.7) { $\mathbb V(t,y)$};
			\draw [thick] (6,5.5) - - (4,8.5);

			\draw[dashed] (0,10) - -  (8,10);  
			\draw[dashed] (3,14) - - (11,14);   
			\draw[dashed] (0,10) - -  (3,14);  
			\draw[dashed] (8,10) - - (11,14);
			\node [font=\footnotesize] at (8,12) { $\mathbb V(s,x)$};
			\draw [thick] (1,10.5) - - (9,13.5);

			\draw [thick] (5,4) - - (5,5);
			\draw[dashed] (5,5.1) - - (5,6.9);
			\draw [thick] (5,7) - - (5,10);
			\draw[dashed] (5,10.1) - - (5,12);
			\draw [thick] (5,12) - - (5,15);
			\node [font=\footnotesize] at (6.5,15) { $\mathbb V(x,y)$};

			\end{tikzpicture} 
			\caption{The above set $X$.} \label{fig0}
		\end{figure}
	\end{center}
\end{example}
Given a set of $n$ lines $X$ and  positive integers $m_1,\ldots, m_n$, we let $Z$ denote the subscheme of~$\mathbb P^1\times \mathbb P^2$ defined by the saturated bihomogeneous ideal
$$I_Z=\bigcap_{L_i\in X} (I_{L_i})^{m_i}\subseteq R.$$
We call $Z$ a set of {\em fat lines} in  $\mathbb P^1\times\mathbb P^2$ whose support is $X=L_1+L_2+\cdots+L_n\subseteq \mathbb P^1\times\mathbb P^2$. The scheme $Z$ will be denoted by $Z=m_1L_1+m_2L_2+\cdots+m_nL_n$.

\begin{remark}\label{r. liaison}

The aim of this paper is to investigate which properties make a set of (fat) lines ACM.
In \cite{FGM2018} and \cite{FM2019} the ACM property (or its failure) was studied for sets of points. Although the ambient spaces in these two papers are different, and the descriptions of the ACM sets of points are certainly different, they have a property in common. Given a set of points in a multiprojective space, there is a complete intersection of points containing it. So its residual is again a set of points.
An ACM set of lines in $\mathbb P^1\times\mathbb P^2$ is a codimension 2 scheme, hence viewed in $\mathbb P^4$ it  is in the same liaison class as a complete intersection. One can hope that liaison tricks will continue to work here. One powerful trick is the result of Gaeta that says that if $X \subset \mathbb P^N$ is ACM of codimension two then one can link in a finite number of steps, always using complete intersections that are generated by minimal generators of the ideal, in such a way that at each step the number of minimal generators drops by one, and the end result is a complete intersection. Conversely, if such a sequence of links exists then $X$ is ACM. If $X \subset \mathbb P^1 \times \mathbb P^2$ is viewed in $\mathbb P^4$, this can still be done using homogeneous complete intersections. However, if we want to insist that our complete intersections are generated by bihomogeneous polynomials (so all varieties in the sequence of links lie in $\mathbb P^1 \times \mathbb P^2$), it might
not be possible. 
(See for instance \cite{Jbook, MN}.) This is illustrated in the following example.
\end{remark}

\begin{example}\label{e. liaison}
Let 
\[
X = \{ \mathbb V(x,y), \mathbb V(s,x), \mathbb V(s,y), \mathbb V(t,x), \mathbb V(t,y), \mathbb V(t,z), \mathbb V(t,x+y+z), \mathbb V(s+t,y) \}.
\]

 \begin{center}
 	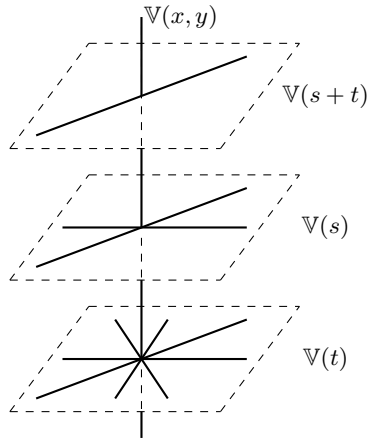
\begin{figure}[H]
 		\centering
 		\begin{tikzpicture}[scale=0.35]
 		
 		\draw[dashed] (0,10) - -  (8,10);  
 		\draw[dashed] (3,14) - - (11,14);   
 		\draw[dashed] (0,10) - -  (3,14);  
 		\draw[dashed] (8,10) - - (11,14);     		
 		\node [font=\footnotesize] at (12,12) { $\mathbb V(s+t)$};
 		\draw [thick] (1,10.5) - - (9,13.5);
 		
 		\draw[dashed] (0,0) - -  (8,0);  
 		\draw[dashed] (3,4) - - (11,4);   
 		\draw[dashed] (0,0) - -  (3,4);  
 		\draw[dashed] (8,0) - - (11,4); 
 		\node [font=\footnotesize] at (12,2) { $\mathbb V(t)$};
 		\draw [thick] (2,2) - - (9,2);
 		\draw [thick] (1,0.5) - - (9,3.5);
 		\draw [thick] (4,0.5) - - (6,3.5);
 		%
 		\draw [thick] (6,0.5) - - (4,3.5);
 		%

 		\draw[dashed] (0,5) - -  (8,5);  
 		\draw[dashed] (3,9) - - (11,9);   
 		\draw[dashed] (0,5) - -  (3,9);  
 		\draw[dashed] (8,5) - - (11,9);
 		\node [font=\footnotesize] at (12,7) { $\mathbb V(s)$};
 		\draw [thick] (2,7) - - (9,7);
 		\draw [thick] (1,5.5) - - (9,8.5);

 		\draw [thick] (5,-1) - - (5,0);
 		\draw[dashed] (5,0.1) - - (5,1.9);
 		\draw [thick] (5,2) - - (5,5);
 		\draw[dashed] (5,5.1) - - (5,6.9);
 		\draw [thick] (5,7) - - (5,10);
 			\draw[dashed] (5,10.1) - - (5,14.9);
 		\draw [thick] (5,12) - - (5,15);
 		\node [font=\footnotesize] at (6.5,15) { $\mathbb V(x,y)$};

 		\end{tikzpicture} 
 		\caption{A representation of the above set $X$.} \label{fig1}
 	\end{figure}
 \end{center}

Then one can check using Macaulay2 or CoCoA that $X$ is ACM, with four minimal generators:
\[
txy,\ sty,\ x^2yz + xy^2z + xyz^2,\ \hbox{ and } s^2tx + st^2x. 
\]
These have bidegree, respectively, (1,2), (2,1), (0,4) and (3,1). In $\mathbb P^4$ it is clear that one can link $X$ using a complete intersection consisting of homogeneous polynomials of degrees 3 and~4 that are part of a minimal generating set. However, if we want to link using bihomogeneous polynomials that are minimal generators, we claim that no such link exists. Indeed, note first that 
the given generators have pairwise common factors.  A bihomogeneous minimal generator of degree 3 has to be one of the two given generators, since any linear combination is no longer bihomogeneous. To get a bihomogeneous minimal generator of degree 4, the only possibility is a polynomial  $L \cdot sty + (s^2tx + st^2x)$ for some linear form $L$ of type (1,0). Note that $st$ would be a factor of any such  bihomogeneous form. Thus any such bihomogeneous polynomial has a factor in common with both $txy$ and $sty$, so no such link is possible.

One could  envision an approach wherein we begin with a set of lines in $\mathbb P^1 \times \mathbb P^2$ and we somehow link  in $\mathbb P^4$ without regard to having the residual be viewable as a subvariety of $\mathbb P^1 \times \mathbb P^2$, looking only to whether or not we can arrive at a complete intersection (hence $X$ is ACM). We have not seen how to make such an approach work.
\end{example}

 
\section{ACM sets of lines in $\mathbb P^1\times \mathbb P^2$: the starting case}\label{s. starting case}
In this section we focus on the ACM property for a set of fat lines $Z$ in $\mathbb P^1\times \mathbb P^2$ satisfying extra conditions.
From now on we will work under the following hypothesis.

\begin{notation}\label{notation}
	Let $Z=Z_1\cup Z_2$ be the partition of a set $Z$ of fat lines  in $\mathbb P^1\times \mathbb P^2$ into horizontal and vertical lines respectively.
	Throughout this section we will assume that
	
\begin{itemize}
	\item[(a)] $Z_1$ is a non-empty set of   horizontal reduced lines;
	\item[(b)] two different horizontal lines of $Z$ are not contained in a hyperplane defined by a form of degree $(0,1)$, i.e., if $\mathbb V(A,B), \mathbb V (A',B)\in Z$, then $(A)=(A').$  
\end{itemize}   
\end{notation}

\begin{remark}\label{r. Z2 fat points P2}
	Note that a set of minimal generators of $I_{Z_2}$ is only in the variables $x,y,z$ so $I_{Z_2}$ is the cone ideal of a set of (fat) points $Y_2$ in $\mathbb P^2.$ 
\end{remark}

\begin{remark} \label{lower star}
In this section we will use the connection between sheafification of a graded module and the direct sum over all twists of the cohomology of the sheaf. This has a version from local cohomology, but we will use the sheaf version. We refer to Hartshorne \cite{hartshorne} for sheafification, and to \cite[Section 1.1]{Jbook}  for the rest, but we recall the basic facts here that we will use in this section.

If $M$ is a graded $R$-module then $\tilde M$ denotes its sheafification. If $\mathcal F$ is a sheaf on $\mathbb P^n$ and $i \geq 0$ is an integer then 
\[
H^i_*(\mathcal F) = \bigoplus_{t \in \mathbb Z} H^i(\mathbb P^n, \mathcal F(t)).
\]
This is a graded $R$-module. In particular, if $I$ is a homogeneous ideal and $\mathcal I = \tilde I$ then $I = H^0_*(\mathcal I)$ if and only if $I$ is already saturated; in general, the module on the right is the saturation of $I$. If 
\[
0 \rightarrow M_1 \rightarrow M_2 \rightarrow M_3 \rightarrow 0
\]
is a short exact sequence of graded modules then 
\[
0 \rightarrow \tilde M_1 \rightarrow \tilde M_2 \rightarrow \tilde M_3 \rightarrow 0
\]
is a short exact sequence of sheaves, but this is only left exact on global sections in general, and we get a long exact on cohomology.

\end{remark}

\begin{remark}\label{r. CI Z1}
	Let $Z$ be as in Notation \ref{notation}.  If $|\pi_1(Z_1)|=1$ then all the horizontal lines of $Z$ are contained in the same plane, say $\mathbb V(A).$ Hence $Z_1$ is a complete intersection of codimension 2. Indeed, say $Z_1=\mathbb V(A,B_1)\cup\cdots \cup \mathbb V(A,B_N)$, we note that the ideal defining $Z_1$ is $I_{Z_1}=(A, F),$ where  $F=B_1B_2 \cdots B_N$. Since $\deg A=(1,0)$ and  $\deg F=(0,N)$ the ideal $(A,F)$ is generated by a regular sequence.  
\end{remark}

The next result make evident what connection there is between the saturation problem studied in Section \ref{s.P2} and the ACM property of $Z$.

\begin{theorem}\label{t. char ACM 1 plane} Let $Z$ be as in Notation \ref{notation} and assume $|\pi_1(Z_1)|=1$.  Let $F$ be of the form in Remark \ref{r. CI Z1}. Then
	$Z$ is ACM if and only if $I_{Z_2}+(F)$ is saturated in $k[x,y,z].$
\end{theorem}
\begin{proof} 

Notice that if $I_{Z_2} + (F)$ is artinian then it is not saturated, so the last condition includes the statement that $I_{Z_2}+(F)$ has height 2 in $k[x,y,z]$.

We look at $Z$, $Z_1$ and $Z_2$ as unions of planes in $\mathbb P^4.$
Consider the short exact sequence
\begin{equation}\label{eq.ses}
	0\to I_Z\to I_{Z_1}\oplus I_{Z_2}\to I_{Z_1}+ I_{Z_2}\to 0.
\end{equation}
	Since $|\pi_1(Z_1)| = 1$, we have  $I_{Z_1}=(A, F)$, a complete intersection (see Remark \ref{r. CI Z1}).
	In particular, note that $Z_1$ and $Z_2$ are ACM unions of planes in $\mathbb P^4$ (see Remark \ref{r. Z2 fat points P2}).

	Sheafifyng the exact sequence (\ref{eq.ses}) and taking  cohomology, we obtain the following exact diagrams:
	
\begin{equation} \label{eq.sheaf ses}
\begin{array}{ccccc}
0 \to I_Z \to I_{Z_1} \oplus I_{Z_2} & \longrightarrow &  [I_{Z_1} + I_{Z_2}]^{sat} \to  H_*^1(\mathcal I_Z) \to 0  \\ 
 & \searrow \hspace{.3in} \nearrow   \\ 
 & I_{Z_1} + I_{Z_2} \\
  & \hspace{-.4in} \nearrow & \hspace{-.3in} \searrow \hfill \\
 \hspace{1.3in} 0 && 0 \hfill
 \end{array}
 \end{equation}
 and
\begin{equation} \label{eq.sheaf ses2}
 \begin{array}{ccccc}
 0 \to H_*^1(\mathcal I_{Z_1} + \mathcal I_{Z_2}) \to  H_*^2(\mathcal I_Z)\to 0.
\end{array}
\end{equation}
Recall that $F \in k[x,y,z]$.  Recall also that $Z$ is ACM if and only if $H^1_*(\mathcal I_Z) = H^2_*(\mathcal I_Z) = 0$. 
Note that the form $A$ is a regular element in $\dfrac{R}{I_{Z_2}+(F)}$ and recall that
\[
I_{Z_1} + I_{Z_2} = (A,F) + I_{Z_2}.
\] 
Suppose first that $F$ does not vanish on any component of $Z_2$. Then $F + I_{Z_2}$ is an unmixed (in particular saturated) height 3 ideal in $R$, and $(A,F) + I_{Z_2} = I_{Z_1} + I_{Z_2}$ is a saturated ideal of height 4 in $R$.
From (\ref{eq.sheaf ses2}) we know that 
\[
\hbox{$H^2_*(\mathcal I_Z) = 0$ \text{ if and only if}  $H^1_* (\mathcal I_{Z_1} + \mathcal I_{Z_2}) = 0 $}.
\]
Since in the current situation $I_{Z_1} + I_{Z_2}$ defines a zero-dimensional scheme, the first cohomology module $H^1_*(\mathcal{ I}_{Z_1}+\mathcal{ I}_{Z_2})$ does not vanish. Thus $Z$ has no hope of being ACM.

What we have just shown is that if $H^2_*(\mathcal I_Z) = 0$ then $I_{Z_2} + (F)$ has height 2 (either in $R$ or in $k[x,y,z]$). We claim that the converse is also true. Indeed, if $I_{Z_2} + (F)$ has height 2 in~$k[x,y,z]$, its saturation defines a zero-dimensional scheme in $\mathbb P^2$, which is automatically ACM. Thus viewed in $R$, this saturation defines an ACM surface in $\mathbb P^4$, so the (not necessarily saturated) ideal $I_{Z_1} + I_{Z_2}$ defines the hyperplane section, which is also ACM. But the first cohomology of the ideal sheaf does not depend on whether the original ideal was saturated or not, so it vanishes and hence $H^2_*(\mathcal I_Z) = 0$ by (\ref{eq.sheaf ses2}).

Thus we can assume without loss of generality that $I_{Z_2} + (F)$ has height 2 (either in $R$ or in~$k[x,y,z]$) and that $H^2_*(\mathcal I_Z) = 0$, and we focus on $H^1_*(\mathcal I_Z)$.
From (\ref{eq.sheaf ses}) we obtain

\medskip

\begin{tabular}{ccllcc}
$H^1_*(\mathcal I_Z) = 0$ & if and only if & $I_{Z_1} + I_{Z_2}$ is saturated in $R$ \\
& if and only if & $I_{Z_2} + (F)$ is saturated in $k[x,y,z]$. \\
\end{tabular}

\medskip

\noindent 
(As before, an ideal in $k[x,y,z]$ of height 2 is Cohen-Macaulay if and only if it is saturated if and only if it is unmixed.)
\end{proof}

\begin{remark}\label{using result insection 2}
Theorem \ref{t. char ACM 1 plane} shows that the ACM property of a set of  (fat) lines $Z=Z_1\cup Z_2$ where $|\pi_1(Z_1)|~=~1$  only depends on the saturation of the ideal $(F)+I_{Z_2}\subseteq k[\mathbb P^2]$. In $\mathbb P^2$, the set $Z_2$ can be viewed as a set of points and the form $F$ defines a curve, $\mathcal C$, that is a union of lines. So, if the hypotheses of Proposition \ref{p.points in P2 and h-vector} are satisfied, the ACM property only depends on the $h$-vectors of $I_{Z_2}$, $I_{Y_1}=(F)+I_{Z_2}$ and $I_{Y_2}=   I_{Z_2} : I_{\mathcal C}.$  
More precisely, $Z$ is ACM if and only if $h_{Z_2}(\tau) =h_{Y_1}(\tau) + h_{Y_2}(\tau-d)$ where $d$ is the degree of $F\in k[\mathbb P^2].$ 
\end{remark}

We introduce the following definition.	Recall that $s$ and $t$ are the variables of degree $(1,0)$ in the coordinate ring of $\mathbb P^1 \times \mathbb P^2$, $R=k[s,t,x,y,z]$.
\begin{definition}\label{d. hat Z}
	Let $Z$ be a set of lines in $\mathbb P^1\times \mathbb P^2$ as in Notation \ref{notation}.  
	We define $$\hat Z=\{\mathbb V(s,B)\ |\ \mathbb V(A,B) \in Z_1\ \text{for some}\ A \}\cup Z_2.$$ 
\end{definition}
\noindent The set $\hat Z$ consists of all the vertical lines in $Z$, together with the projections  of all the horizontal lines in $Z$ down to the hyperplane defined by $\mathbb V(s)$.  See Figure $\ref{fig:Z vs hat Z}$ and Figure $\ref{fig:X vs hat X}$ for two examples. 

Note that in Definition \ref{d. hat Z}  the set  $\hat Z$ satisfies the hypotheses of Theorem \ref{t. char ACM 1 plane}; indeed we have $|\pi_1 (\hat Z_1)|=1.$

\begin{proposition}\label{p. Z ACM -> hat Z ACM}
	Let $Z$ be as in Notation \ref{notation}. If $Z$ is ACM then $\hat Z$ is ACM, and $Z$ and $\hat Z$ share the same multigraded homological invariants. 
\end{proposition}

\begin{proof}

Without loss of generality we can assume that the linear form $t$ (one of the indeterminates)  is a regular element in $R/I_Z$. If we look at $Z$ and $\hat Z$ as unions of planes in $\mathbb P^4$, the hyperplane defined by $t$ does not contain any component either of $Z$ or of $\hat Z$, and so meets each such component in a line.

Let us examine the effect on both kinds of components of $Z$. A ``horizontal" component is defined by an ideal of the form $(\ell_i, m_i)$ with $\ell_i \in k[s,t]$ and $m_i \in k[x,y,z]$, so the hyperplane section defined by the linear form $t$ is an ideal of the form $(\ell_i, m_i, t) = (s,t,m_i)$. A ``vertical" component is defined by an ideal of the form $(L_j,M_j)$, with $L_j,M_j \in k[x,y,z]$, so the hyperplane section by $t$ is defined by~$(L_j,M_j,t)$.  

Since $Z$ is ACM (viewed in $\mathbb P^4$), we have that $I_Z + (t)$ is the saturated ideal of the hyperplane section (a union of lines), which is ACM. 
The entire hyperplane section by $t$  then has the saturated ideal
\[
 I_Z + (t)  =  \bigcap_i (s, m_i,t) \  \cap \ \bigcap_j (L_j, M_j, t),
\]
since $Z$ is ACM, and defines an ACM curve in $\mathbb P^4$.

On the other hand, up to saturation the latter is also the hyperplane section by $t$ of $\hat Z$. Since the  curve is ACM, the union of planes $\hat Z$ must also be ACM by \cite[Proposition 2.1]{HU}. Thus 
\[
I_{\hat Z} + (t) = \Big( \bigcap_i (s, m_i) \cap \bigcap_j (L_j, M_j) \Big) + (t). 
\]
Finally, since $Z$ and $\hat Z$ are both ACM with the same hyperplane section, they must in fact have the same homological invariants.

Up to this point we have only shown that $Z$ and $\hat Z$ have the same {\em graded} homological invariants, by viewing $Z$ and $\hat Z$ as subschemes of $\mathbb P^4$. But in fact they began in $\mathbb P^1 \times \mathbb P^2$, so $I_Z$ and $I_{\hat Z}$ have multigraded Betti numbers and in particular multigraded minimal generators. When we reduce by the non-zerodivisor $t$, this preserves the multigrading. Hence we have the result.
\end{proof}

The following  is an immediate consequence of Proposition \ref{p. Z ACM -> hat Z ACM}. It gives us a necessary condition for the Cohen-Macaulayness of $Z.$ We denote by $\beta_{0,(a,b)}(I_Z)$ the minimal number of generators of $I_Z$ of degree $(a,b).$

\begin{corollary}\label{c.Betti ACM}
	Let $Z$ be a set of lines in $\mathbb P^1\times \mathbb P^2$ as in Notation \ref{notation}. If $Z$ is ACM, then $\beta_{0,(a,b)}(I_Z)=0$ for each $a>1.$
\end{corollary}

\begin{proof}

 Indeed this is true for $\hat Z$.
 \end{proof}

\begin{remark}
Corollary \ref{c.Betti ACM} seems very surprising, at first glance. But it assumes from the beginning that $Z$ is ACM, and then that it satisfies condition (b) of Notation \ref{notation}. Neither of these is particularly restrictive by itself, but the point is that the combination is restrictive. Still, having both conditions does not by any means imply that $Z = \hat Z$. For instance, consider Example \ref{three lines}. One might think at first that there is a minimal generator of type $(2,1)$ (e.g. $stx$), but in fact one can check that the ideal has two minimal  generators of bidegree $(1,1)$ and one of bidegree $(0,2)$.

Furthermore, starting with a $Z$ that is ACM and satisfies Notation \ref{notation}, it is easy to use basic double linkage to produce a new $Z'$ that is ACM and does not satisfy $\beta_{0,(a,b)}(I_Z)=0$ for each $a>1$, but we lose the property given in Notation \ref{notation} (b). For instance, returning to Example~\ref{three lines}, one could form $(s+t) \cdot I_Z + (xy)$, which is ACM and has a minimal generator of bidegree (2,1); but now it does not satisfy Notation \ref{notation}.

We do not know if the converse of Corollary \ref{c.Betti ACM} is true. So we pose the following question:
\begin{quotation}
	{\em Let $Z$ be as in Notation \ref{notation}. Assume $\hat Z$ is ACM and $\beta_{0,(a,b)}(I_Z)=0$ for each $a>1.$ Is $Z$ ACM? }
\end{quotation}
\noindent However, the next example shows that, without the condition on $\beta_{0,(a,b)} (I_Z)$,  $\hat Z$ ACM does not imply $Z$ ACM.

\end{remark}

\begin{example}\label{e. Z not ACM hat Z ACM}
Let $Z$ the set of lines in $\mathbb P^1\times \mathbb P^2$ defined by the ideal $I_Z=(t,x)\cap (s,y)$. The set $Z$ is clearly not ACM. However, $\hat Z$, defined by $I_{\hat Z}=(s,x)\cap (s,y)$, is ACM. 
	
\begin{figure}[h]
	\centering
	\subfloat[The set $Z$.]{
		\begin{tikzpicture}[scale=0.33]
	\draw[dashed] (0,5) - -  (8,5);  
	\draw[dashed] (3,9) - - (11,9);   
	\draw[dashed] (0,5) - -  (3,9);  
	\draw[dashed] (8,5) - - (11,9); 
	\node [font=\footnotesize] at (3,6.6) { $\mathbb V(s,y)$};
	\draw [thick] (6,5.5) - - (4,8.5);
	
	\draw[dashed] (0,10) - -  (8,10);  
	\draw[dashed] (3,14) - - (11,14);   
	\draw[dashed] (0,10) - -  (3,14);  
	\draw[dashed] (8,10) - - (11,14);
	\node [font=\footnotesize] at (8,12) { $\mathbb V(t,x)$};
	
	\draw [thick] (1,10.5) - - (9,13.5);
	\end{tikzpicture}	\label{fig:01}}		\qquad
	\subfloat[The set $\hat Z$.]{
		\begin{tikzpicture}[scale=0.42]
	
	\node [font=\footnotesize] at (3,12) { $\mathbb V(s,y)$};
	\draw [thick] (6,10.5) - - (4,13.5);
	
	\draw[dashed] (0,10) - -  (8,10);  
	\draw[dashed] (3,14) - - (11,14);   
	\draw[dashed] (0,10) - -  (3,14);  
	\draw[dashed] (8,10) - - (11,14);
	\node [font=\footnotesize] at (8,12) { $\mathbb V(s,x)$};
	
	\draw [thick] (1,10.5) - - (9,13.5);

	\end{tikzpicture} 
		\label{fig:02}}
	\caption{ Example \ref{e. Z not ACM hat Z ACM}.}
	\label{fig:Z vs hat Z}
\end{figure}

\end{example}

From Proposition \ref{p. Z ACM -> hat Z ACM} we know that the ACM property of $\hat Z$ is   a necessary condition for the ACM property of $Z$. 
Example \ref{e. Z not ACM hat Z ACM} shows that it is not sufficient. The following lemma will give us another necessary condition.

\begin{lemma}\label{l. v-connected}
	Let $Z$ be an ACM set of lines as in Notation \ref{notation}. Let $\mathbb V(A,B), \mathbb V(A',B')\in Z$ be two horizontal lines and $A\neq A'$. Then $\mathbb V(B,B')\in Z.$  
\end{lemma}
\begin{proof}
	Look at $Z$ as a set of planes in $\mathbb P^4$. If $\mathbb V(B,B')\notin Z$ then  $Z$ is not locally Cohen-Macaulay at the point defined by the ideal $\mathfrak p=(A,A',B,B')$, hence $Z$ is not ACM. Contradiction.
\end{proof}

Then it is natural to give the next definition.
\begin{definition} Let $Z$ be as in Notation \ref{notation}. We say that $Z$ is \textit{v-connected} if, for each $\mathbb V(A,B), \mathbb V(A',B')\in Z_1$ where $A\neq A'$ and $B\neq B'$, we have $\mathbb V(B,B')\in~Z_2$.
\end{definition}

In the next example we show that this property is still not enough to ensure the ACM property.
\begin{example}\label{e. Z ACM X not ACM}
Let $X$ be the set of lines in $\mathbb P^1 \times \mathbb P^2$ defined by the ideal
$$
I_X= (s,x)\cap (s,y) \cap (t,x+y)\cap (x,y).
$$ 
Note that $X$ is v-connected. Moreover, from Theorem \ref{t. char ACM 1 plane}, the set $\hat X$ defined by
$$I_{\hat X}= (s,x)\cap (s,y) \cap (s,x+y)\cap (x,y)$$
is ACM.

\begin{figure}[h]
	\centering
	\subfloat[The set $X$.]{
 		\begin{tikzpicture}[scale=0.35]

\draw[dashed] (0,5) - -  (8,5);  
\draw[dashed] (3,9) - - (11,9);   
\draw[dashed] (0,5) - -  (3,9);  
\draw[dashed] (8,5) - - (11,9); 
\node [font=\footnotesize] at (12,7) { $\mathbb V(t)$};

\draw [thick] (6,5.5) - - (4,8.5);
%

\draw[dashed] (0,10) - -  (8,10);  
\draw[dashed] (3,14) - - (11,14);   
\draw[dashed] (0,10) - -  (3,14);  
\draw[dashed] (8,10) - - (11,14);
\node [font=\footnotesize] at (12,12) { $\mathbb V(s)$};
\draw [thick] (2,12) - - (9,12);
\draw [thick] (1,10.5) - - (9,13.5);

\draw [thick] (5,4) - - (5,5);
\draw [dashed] (5,5) - - (5,7);
\draw [thick] (5,7) - - (5,10);
\draw [dashed] (5,10) - - (5,12);
\draw [thick] (5,12) - - (5,15);
\node [font=\footnotesize] at (6.5,15) { $\mathbb V(x,y)$};

\end{tikzpicture} 	\label{fig2}}		\qquad
	\subfloat[The set $\hat X$.]{
	 		\begin{tikzpicture}[scale=0.45]

	\draw [thick] (6,10.5) - - (4,13.5);
	%

	\draw[dashed] (0,10) - -  (8,10);  
	\draw[dashed] (3,14) - - (11,14);   
	\draw[dashed] (0,10) - -  (3,14);  
	\draw[dashed] (8,10) - - (11,14);
	\node [font=\footnotesize] at (12,12) { $\mathbb V(s)$};
	\draw [thick] (2,12) - - (9,12);
	\draw [thick] (1,10.5) - - (9,13.5);

	\draw [thick] (5,8) - - (5,10);
	\draw[dashed] (5,10) - - (5,12);
	\draw [thick] (5,12) - - (5,15);
	\node [font=\footnotesize] at (6.5,15) { $\mathbb V(x,y)$};
	\end{tikzpicture} 
		\label{fig3}}
	\caption{ Example \ref{e. Z ACM X not ACM}.}
	\label{fig:X vs hat X}
\end{figure}

According to CoCoA, $X$ is not ACM.
A computer experiment shows that the set of fat lines, $Z$,  whose ideal is
$$I_{Z}= (s,x)\cap (s,y) \cap (t,x+y)\cap (x,y)^2$$
is ACM.
Thus that $Z$ is a  (non-reduced) ACM set of lines whose support $X$ is not ACM. 
The next result explores this idea further.
\end{example}

\begin{theorem}\label{p. fully v-connected 1}
	Let $Z=\mathbb{V}(A_1,B_1)+ \cdots+ \mathbb{V}(A_n,B_n)+ m\cdot\mathbb{V}(B_1,B_2) $ be a set of lines of~$\mathbb P^1\times \mathbb P^2$ such that
	\begin{itemize}
		\item[(1)] $\mathbb{V}(A_{n-1})\neq \mathbb{V}(A_n)$ (i.e. not all the horizontal lines are in the same plane);
		\item[(2)] $\mathbb{V}(B_i)\neq \mathbb{V}(B_j)$ for $i\neq j$, (i.e. $B_1, \cdots, B_n$ define different planes); 
		\item[(3)] $ \mathbb{V}(B_1, \ldots, B_n)= \mathbb{V}(B_1, B_2)$, (i.e. $B_1, \ldots, B_n$ define planes in the pencil containing $\mathbb{V}(B_1, B_2)$, therefore $\mathbb{V}(B_1, B_2)=\mathbb{V}(B_i, B_j)$ for all $i\neq j$).
	\end{itemize}    
	Then $Z$ is ACM if and only if $m\ge n-1$.
\end{theorem}

\begin{proof} 
First assume $n=2$, i.e.,  $Z=\mathbb{V}(A_1,B_1)+\mathbb{V}(A_2,B_2)+ m\cdot\mathbb{V}(B_1,B_2)$.  
	\begin{itemize}
		\item[$(\Rightarrow)$] By Lemma \ref{l. v-connected}, $Z$ ACM implies  $m\ge 1$.
		\item[$(\Leftarrow)$] 	On the other hand, take $m\ge 1$ and consider the following short exact sequences
		\begin{equation}\label{eq.1}
		0 \to I_Z \to (A_1,B_1) \oplus \left [ (A_2,B_2) \cap (B_1,B_2)^m \right ] \to \left  (A_1,B_1)+[(A_2,B_2)\cap(B_1,B_2)^m \right ] \to 0,
		\end{equation} 
and
		\begin{equation}\label{eq.2}
		0\to (A_2,B_2)\cap(B_1,B_2)^m\to (A_2,B_2)\oplus(B_1,B_2)^m\to (A_2,B_2)+(B_1,B_2)^m\to 0.
		\end{equation} 
		
\noindent Look at \eqref{eq.2} and note that $(A_2,B_2)+(B_1,B_2)^m=(A_2,B_2,B_1^m)$. Thus, the ideal $(A_2,B_2)+(B_1,B_2)^m$ defines an ACM scheme of height 3. 
		Since the schemes defined by $(A_2, B_2)$ and $(B_1,B_2)^m$ are ACM of height 2, then a mapping cone argument gives that the ideal  $(A_2,B_2)\cap(B_1,B_2)^m$ has projective dimension 2.
		Now consider~\eqref{eq.1}, and note that 
		\[
		(A_1,B_1)+ \left [ (A_2,B_2)\cap(B_1,B_2)^m \right ] = (A_1,B_1,B_2^m),
		\] 
so it defines an ACM scheme of height 3. Moreover, the module in the middle of the sequence has projective dimension 2, hence  $Z$ is ACM. 
\end{itemize}
	
\noindent Now we proceed by induction, the base case having just been proven. Assume that the statement is true for $n-1 \geq 1$ and assume that $n \geq 3$. That is, assume that if $Z$ has $n-1$ reduced components, then $Z$ is ACM if and only if $m \geq n-2$.
Set 
\[
Z'=\mathbb{V}(A_2,B_2)+ \cdots+ \mathbb{V}(A_n,B_n)+ m\mathbb{V}(B_1,B_2)
\]
and consider the following short exact sequence:
\begin{equation}\label{eq.3}
0\to I_Z \to (A_1,B_1)\oplus I_{Z'}\to (A_1,B_1)+I_{Z'}\to 0.
\end{equation}

\begin{itemize}

\item[$(\Leftarrow)$] 
		If $m\ge n-1$ then, by the inductive hypothesis, $Z'$ is ACM.
		Thus, in order to prove that $Z$ is ACM, by the sequence \eqref{eq.3}, it is enough to show that $(A_1,B_1)+I_{Z'}$ defines an ACM variety of height 3. 
		We show that $(A_1,B_1)+I_{Z'}=(A_1,B_1,B_2^m).$ 
	\begin{itemize}
			\item[$\bullet$] $(A_1,B_1)+I_{Z'}\subseteq(A_1,B_1,B_2^m)$. Indeed, if $F\in I_{Z'}$ then $F\in (B_1,B_2)^m\subseteq (B_1,B_2^m)$;
			\item[$\bullet$] $(A_1,B_1)+I_{Z'}\supseteq(A_1,B_1,B_2^m)$. We first claim that the form $F=B_2^{m-n+2}\cdot B_3\cdots B_n$ belongs to $I_{Z'}$. By hypotheses $2)$ and $3)$,  for $j>3$ we have $B_j=\lambda_jB_1+\mu_jB_2$, where $\lambda_j,\mu_j\neq 0$. Thus, $F=F'B_1+aB_2^m\in I_{Z'}$ and therefore $B_2^m\in (A_1,B_1)+I_{Z'}.$  
	\end{itemize}
\vspace{.1in}

\item [$(\Rightarrow)$]
Assume that $Z$ is ACM. View $Z$ as a set of planes  in $\mathbb P^4$. By the sequence \eqref{eq.3}, the ideal $(A_1,B_1)+I_{Z'}$ is  saturated in $k[\mathbb P^4]$. Without loss of generality we can assume the linear form $z$ corresponds to a general form of degree one in  $k[x,y,z]$, so that the intersection of $Z$ with $\mathbb V(z)$ is a proper hyperplane section of $Z$ in $\mathbb P^4$, all but one component of $Z\cap\mathbb V(z)$ are reduced, and  $(\bar B_1 ), \ldots , (\bar B_n )$ are all different in $R/(z)$.
Denoting by $J$ the ideal defining $Z\cap\mathbb V(z)$ in $\mathbb P^3$, we have that 
$J \cong (I_Z+(z))/(z)$ and it is Cohen-Macaulay in  $k[\mathbb P^3]\cong R/(z)$. Let us denote by $Y$ the scheme defined by $J$ in $\mathbb P^3$. Hence $Y$ is an ACM set of lines (one non-reduced) in $\mathbb P^3$; precisely 
\[
J= (A_1,\bar B_1)\cap \cdots \cap (A_n,\bar  B_n)\cap (\bar  B_1,\bar  B_2)^m
\]
where $\mathbb V(\bar  B_j)= \mathbb V(B_j)\cap \mathbb V(z)$. (Recall that $A_i \in k[s,t]$ while $B_j \in k[x,y,z]$, so abusing notation we will view $A_i \in k[\mathbb P^3]$.)  We set  
\begin{equation} \label{shorter pf}
J'= (A_2,\bar  B_2)\cap \cdots \cap (A_n,\bar  B_n)\cap (\bar  B_1,\bar  B_2)^m 
\end{equation}
and denote by $Y'$ the scheme defined by $J'$, and by $\lambda$ the line defined by $(A_1, \bar B_1)$.
Consider the short exact sequence
\begin{equation}\label{eq.J}
0 \to J \to (A_1,\bar B_1)\oplus J'\to (A_1,\bar B_1)+J'\to 0.
\end{equation}
Since $J$ defines an ACM curve $Y$, sheafifying and taking the long exact sequence in cohomology gives us that the ideal  $(A_1,\bar  B_1) +J'$ is saturated in $k[\mathbb P^3]$; in particular it has height 3, and 
\begin{equation} \label{key eqn}
(A_1, \bar B_1) + J' = (A_1, \bar B_1, \bar B_2^u).
\end{equation}

What is $u$? We have by (\ref{key eqn}) that
$u$ is the minimum such that $[(A_1 , \bar B_1 ) + J']_{(0,u)} \neq [(A_1 ,\bar B_1 )]_{(0,u)}$ . 
Since $\bar B_2 \bar B_3 \ldots \bar B_n$ is the only minimal generator of 
$(A_ 2 , \bar B_2 )\cap \ldots \cap (A_n , \bar B_n )$ with bidegree $(0, v)$, 
the equation (\ref{shorter pf})  implies that $u = \max(n - 1, m).$ In particular, $u \geq n-1$.

Now we assume by contradiction that $m<n-1.$ Consider the two forms
\[
\begin{array}{ccl}
F_1 & = & \bar B_2\cdots \bar B_{n-2} \bar B_{n-1} A_n \ \hbox{  and} \\

F_2 & = & \bar B_2\cdots \bar B_{n-2} A_{n-1}\bar B_n.
\end{array}
\]
Both $F_1$ and $F_2$ have degree $(1,n-2)$ and both belong to $J' $,  and so, 
\[
F_1, F_2 \in  (A_1, \bar B_1) + J' = (A_1, \bar B_1, \bar B_2^u).
\]
By hypothesis (2), $F_1, F_2\notin (\bar B_1).$
By hypothesis (1), at least one among $F_1$ and $F_2$ does not belong to $(A_1)$, say $F_1$.   This means that $F_1\in (\bar B_1, \bar B_2^u)$ and hence $\bar B_2\cdots \bar B_{n-1} \in (\bar B_1, \bar B_2^u)$, so $u\le n-2$ which contradicts what we have shown above.
Therefore we have shown $m \geq n-1$ as desired.\end{itemize} \end{proof}

As a corollary, 
Theorem \ref{p. fully v-connected 1} can be translated to a result about lines in $\mathbb P^3$.

\begin{corollary} \label{cor:curve P3}
Let $\lambda_1$ and $\lambda_2$ be two skew lines in $\mathbb P^3$.
	Let $\alpha_1,\ldots \alpha_n$ and $\beta_1,\ldots \beta_n$ be planes containing $\lambda_1$  and  $\lambda_2$ respectively, such that
	\begin{itemize}
		\item $\beta_i\neq \beta_j$ for any $i\neq j$; 
		\item  $\alpha_{n-1}\neq \alpha_{n}.$
	\end{itemize}
	  Set $\ell_i=\alpha_i\cap \beta_i$. Then, the set of lines $Y=\ell_1+ \cdots +\ell_n + m\cdot \lambda_2$ is ACM if and only if~$m\ge n-1$.
\end{corollary}

\begin{proof}
We have $k[\mathbb P^4] = k[s,t,x,y,z]$ as usual. 
In Theorem \ref{p. fully v-connected 1}, viewing the configuration as a union of planes in $\mathbb P^4$ (one not reduced), the $A_i$ are hyperplanes vanishing on the 2-plane defined by $s=t=0$ and (without loss of generality, after possibly a change of variables) the $B_i$ can be taken to be hyperplanes vanishing on the 2-plane defined by $x=y=0$. In our current setting, without loss of generality (after possibly a change of variables) we can take $k[\mathbb P^3] = k[s,t,x,y]$, and set $\lambda_1$ to be the line defined by $s=t=0$ and set $\lambda_2$ to be the line defined by $x=y=0$.  Note that $z$ is a non-zerodivisor for $R/I_Z$ in Theorem \ref{p. fully v-connected 1}. Then $Y$ is the hyperplane section of $Z$ cut out by the hyperplane defined by $z$. By \cite[Proposition 2.1]{HU}, $Y$ is ACM if and only if $Z$ is ACM. The result follows immediately.
\end{proof}

Theorem \ref{p. fully v-connected 1} leads us to introduce a new definition.
\begin{definition}\label{d.fully v connected}
	Let $Z$ be as in Notation \ref{notation}. Assume that whenever $\mathbb V(A,B), \mathbb V(A',B')\in Z_1$ with $A\neq A'$ and $B\neq B'$, the vertical line $V(B,B')$ meets $n$ horizontal lines of $Z$ (where the integer $n \geq 2$ depends on the choice of $B$ and $B'$). 
	We say that $Z$ is \textit{fully v-connected}  if $(n-1)\mathbb V(B,B')\subseteq Z$, i.e., for each such $B, B'$ the line $\mathbb V(B,B')$ is contained in $Z$ with  multiplicity at least $n-1$.  
\end{definition} 

It was shown in \cite[Lemma 4.1]{CFGetc} that any set of points in $\mathbb P^1 \times \mathbb P^1$ is locally a complete intersection, and hence locally Cohen-Macaulay. In the next theorem we characterize the local Cohen-Macaulay property of $Z$ in terms of the fully v-connected property. 

\begin{theorem}\label{t. local CM <-> fully v connected}
	 Let $Z=Z_1\cup Z_2$ be a set of lines of $\mathbb P^1\times \mathbb P^2$ such that  the assumptions of Notation \ref{notation} hold, i.e.
	 \begin{itemize}
	 	\item[(a)] $Z_1$ is a non-empty set of reduced  horizontal lines;
	 	\item[(b)] two different horizontal lines of $Z$ are not contained in a hyperplane defined by a form of degree $(0,1)$, i.e. if $(A,B), (A',B)\in Z_1$ then $(A)=(A').$ 
	 \end{itemize}   
 Then, $Z$ is locally Cohen-Macaulay if and only if $Z$ is fully v-connected.
	
\end{theorem}
\begin{proof}

First assume that $Z$ is locally Cohen-Macaulay. Let $(B,B')^m$ define a vertical fat line $m L$ in $Z$. It is not restrictive to assume that $(B,B')=(x,y)$,  and also that $z$ is a regular element in $R/I_Z$. 
Let $\mathbb V(A_1,B_1), \ldots, \mathbb V(A_n,B_n)\in Z_1$ be the horizontal lines meeting $L$.
Let $P \in \mathbb P^4$ be the point defined by $I_P = (s,t,x,y)$.  In particular, $(x,y)=(B_1, \ldots, B_n),$ where $B_1, \ldots, B_n$ define different hyperplanes. 
Let  $W=\mathbb V(A_1,B_1)+ \cdots+ \mathbb V(A_n,B_n)+mL$ be the corresponding subscheme of $Z$. Note that  $W$  satisfies  the hypotheses of Theorem \ref{p. fully v-connected 1}, so $W$ is ACM if and only if $m\ge n-1$.   
		
Notice that all components of $W$ (viewed in $\mathbb P^4$) contain the point $P$, and no other component of $Z$ does. Thus since the components of $W$ are linear, $W$ is the cone in $\mathbb P^4$ with vertex $P$ over a curve, $C$, in the hyperplane $\mathbb P^3$ defined by $z=0$, and $C$ is precisely of the form given in Corollary \ref{cor:curve P3}.
Now, since $Z$ is locally Cohen-Macaulay, in particular at $P$, we conclude that $C$ is ACM, and hence by Corollary \ref{cor:curve P3} we have that $m \geq n-1$.  Since this holds for each vertical line (noting that $m$ and $n$ will change for different vertical lines), $Z$ is fully v-connected.

Conversely, assume that $Z$ is fully v-connected. Certainly at any smooth point of $Z$ (viewed in $\mathbb P^4$), $Z$ is locally Cohen-Macaulay. Similarly, at the intersection of two or more horizontal lines that do not have a vertical line through the point of intersection, locally it is a complete intersection, hence locally Cohen-Macaulay.

Although the vertical  (possibly fat) lines are disjoint in $\mathbb P^1 \times \mathbb P^2$, the corresponding (fat) planes do meet in $\mathbb P^4$. The intersection of any finite number of such planes is one-dimensional. Still, this is a cone over a set of (fat) points in $\mathbb P^2$, so it is ACM and hence is locally Cohen-Macaulay along this locus.

Consider the intersection in $\mathbb P^1 \times \mathbb P^2$ of a (fat) vertical line of multiplicity $m$ and a collection of $n$ horizontal lines. Without loss of generality assume that the vertical line has support given by $(x,y)$. 	Let  $J=(s,t,x,y)$. Note that the ideal $J$ defines a point $P$ in $\mathbb P^4$. Let  $W=\mathbb V(A_1,B_1)+ \cdots+ \mathbb V(A_n,B_n)+mL$ be the subscheme of $Z$ corresponding to the components, viewed in $\mathbb P^4$, containing $P$. Note that  $W$  satisfies  the hypotheses of Theorem \ref{p. fully v-connected 1}, so $W$ is ACM if and only if $m\ge n-1$. 

Since the components of $W$ are linear, $W$ is the cone in $\mathbb P^4$ with vertex $P$ over a curve, $C$, in the hyperplane $\mathbb P^3$ defined by $z=0$ (which is a non-zerodivisor), and $C$ is precisely of the form given in Corollary \ref{cor:curve P3}.  Since $Z$ is fully v-connected, we know that $m \geq n-1$. Thus, by Corollary \ref{cor:curve P3}, $C$ is ACM. Hence $Z$ is locally Cohen-Macaulay at $P$.
\end{proof}

\begin{remark}\label{r. Z ACM} Let $Z=Z_1\cup Z_2$ be a set of lines of $\mathbb P^1\times \mathbb P^2$ such that the assumptions of Notation~\ref{notation} hold.
We wonder  which hypotheses  ensure $Z$ to be ACM. Note that the fully v-connected condition is not enough. Take for instance $Z=\mathbb V( s, x )\cup\mathbb V( y, z )$; it trivially satisfies Definition \ref{d.fully v connected} and it is not ACM. So, as shown in Proposition \ref{p. Z ACM -> hat Z ACM}, we at least need to ask  $\hat Z$ to be ACM. 
It would be interesting to show if these two conditions,  $Z$ fully v-connected together with $\hat Z$  ACM, are sufficient to guarantee the ACM property for $Z.$
\end{remark}

\end{document}